\documentclass{conm-p-l}

\usepackage{amssymb}

\newtheorem{theorem}{Theorem}[section]
\newtheorem{lemma}[theorem]{Lemma}
\newtheorem{prop}[theorem]{Proposition}

\theoremstyle{definition}

\theoremstyle{remark}
\newtheorem{remark}[theorem]{Remark}

\numberwithin{equation}{section}

\DeclareMathOperator{\GL}{GL}

\DeclareMathOperator{\Sp}{Sp}
\DeclareMathOperator{\PSp}{PSp}
\DeclareMathOperator{\Orth}{O}
\DeclareMathOperator{\SO}{SO}

\DeclareMathOperator{\PSL}{PSL}
\DeclareMathOperator{\PSU}{PSU}

\newcommand{\la}{\langle}
\newcommand{\ra}{\rangle}
\renewcommand{\phi}{\varphi}
\newcommand{\F}{\Bbb{F}_q}
\newcommand{\Fe}{\Bbb{F}_{2^e}}

\newcommand{\rk}{{\rm rk}}
\renewcommand{\leq}{\leqslant}
\renewcommand{\geq}{\geqslant}

\begin{document}

\title[String C-groups for symplectic and orthogonal groups]{On the ranks of string C-group representations for symplectic and orthogonal groups}

\author{Peter A. Brooksbank}
\address{
	Department of Mathematics\\
	Bucknell University\\
	Lewisburg, PA 17837\\
	United States
}
\email{pbrooksb@bucknell.edu}
\thanks{This work was partially supported by a grant from the Simons Foundation (\#281435 to Peter Brooksbank).
}

\subjclass[2000]{Primary }

\date{}

\begin{abstract}
We determine the ranks of string C-group representations of
the groups $\PSp(4,\F)\cong\Omega(5,\F)$, and comment on those of 
higher-dimensional symplectic and orthogonal groups.
\end{abstract}

\maketitle

\section{Introduction}
\label{sec:intro}
Abstract polytopes are incidence structures that generalize classical geometric objects
such as the Platonic solids. The study of these structures (and their associated symmetry 
groups) continues to be a fertile area of research. Abstract {\em regular} 
polytopes and their symmetry groups are fundamentally linked by the notion of a {\em string C-group}, 
namely a group $G$ having a distinguished generating sequence $\rho_0,\ldots,\rho_{n-1}$ 
of involutions satisfying the following two conditions:
\begin{align}
\label{eq:sp}
\forall\;0\leq i<j\leq n-1, && j-i>1~\Longrightarrow~[\rho_i,\rho_j]=1; 
\hspace*{15mm} \\
\label{eq:ip}
\forall \; I,J\subseteq \{0,\ldots,n-1\}, && \la \rho_i\colon i\in I\ra\cap\la \rho_j\colon j\in J\ra=
\la \rho_k\colon k\in I\cap J\ra.
\end{align}
The first condition---the {\em string property}---asserts that $G$
is the quotient of a Coxeter group whose diagram is a string (under the conventional assumption that commuting generators
are not joined by an edge). 
Indeed, if we label each edge in a $n$-node string diagram by the order of the corresponding product $\rho_i\rho_{i+1}$,
and define $\Gamma$ to be the Coxeter group having this diagram, then $G$ is a {\em smooth} quotient of $\Gamma$.
The second condition---known as the {\em intersection
property}---is satisfied by all Coxeter groups; hence, condition~\eqref{eq:ip} requires that the quotient $G$ inherit this
property from its parent group $\Gamma$. The integer $n$ is the {\em rank} of the string C-group $G$.

Each string C-group $G$ has an associated abstract regular polytope $\mathcal{P}(G)$
whose symmetry group is $G$. Conversely, if $G$ is the symmetry group of some abstract regular polytope $\mathcal{P}$, one
can build a generating sequence of involutions in $G$ satisfying~\eqref{eq:sp} and~\eqref{eq:ip}. Thus, the study of abstract
regular polytopes is equivalent to the study of string C-groups~\cite[Section 2]{McS:ARP}, and
in this paper we adopt the group perspective.
For a group $G$, let $\rk(G)$ denote the set of integers $n$
such that $G$ has a string C-group representation of rank $n$. We prove the following:

\begin{theorem}
\label{thm:psp4-main}
Let $\F$ be the finite field with $q$ elements and $G=\PSp(4,\F)$. Then $\rk(G)=\emptyset$ if $q=3$,
and $\rk(G)=\{3,4,5\}$ if $q\neq 3$.
\end{theorem}

Theorem~\ref{thm:psp4-main} may be seen as contributing to an ongoing effort to determine the
ranks of string C-group representations for families of finite simple groups. 

A great deal is known about rank 3 representations: see~\cite{Co:alt1,Co:alt2,Nu:alt-rank3} for
the alternating groups, \cite{Nu:even-rank3,Nu:odd-rank3-1,Nu:odd-rank3-2} for the simple groups
of Lie type, and~\cite{Maz:sporadic} for the sporadic simple groups. The upshot is that most
finite simple groups have string C-group representations of rank 3; Figure 1
lists the exceptions. 

\begin{figure}
\begin{tabular}{||c|c||}
\hline\hline
Group & Reference(s) \\ \hline\hline
${\rm Alt}(6),\,{\rm Alt}(7)$ & \cite{Nu:alt-rank3} \\ \hline
$\PSL(3,\F),\,\PSU(3,\F),\,\PSp(4,\Bbb{F}_3)$ & \cite{Nu:even-rank3,Nu:odd-rank3-1,Nu:odd-rank3-2} \\ \hline
$\PSL(4,\F)~(q\;{\rm even}),\,\PSU(4,\F)~(q\;{\rm even})$ & \cite{Nu:even-rank3} \\ \hline
$\PSU(4,\Bbb{F}_3),\;\PSU(5,\Bbb{F}_2)$ & M. Macaj \& G. Jones \\ \hline
$M_{11},\,M_{22},\,M_{23},\,McL$ & \cite{Maz:sporadic} \\ \hline\hline
\end{tabular}
\caption{Simple groups with no rank 3 string C-group representation. The exceptions of Macaj and Jones,
initially overlooked by Nuzhin, were communicated via D. Leemans.}
\end{figure}

The picture for ranks higher than 4 is far less complete. Some negative results have been proved for
linear groups of low Lie rank: for example, none of $\PSL(3,\F)$, $\PSU(3,\F)$, $\PSL(4,\Fe)$, $\PSU(4,\Fe)$ have 
string C-group representations of any rank~\cite{BV:PSL3,BL:PSL4,BFL:orth}, while 
$4\in \rk(\PSL(2,\F))$ if, and only if, $q\in\{11,19\}$~\cite{LS:PSL2}. 
Figure 2 summarizes the positive results
for infinite families of simple groups and their variants.

\begin{figure}
\label{fig:high-rank}
\begin{tabular}{||c|c|c|c||}
\hline\hline
$G$ & Restrictions & $\rk(G)$ & Reference(s) \\
\hline\hline
${\rm Sym}(n)$ & $n\geq 5$ & $\{3,\ldots,n-1\}$ & \cite{FL:symmetric} \\ \hline
${\rm Alt}(n)$ & $n\geq 12$ & $\{3,\ldots,\lfloor\frac{n-1}{2}\rfloor\}$ & \cite{FL2} \\ \hline
$\Orth^{\pm}(2n,\Fe)$ & $e\geq 2$ & $\{3,\ldots,2n\}$ & \cite{BFL:orth} \\ \hline
$\PSp(2n,\Fe)$ & $e\geq 2$ & $\{3,\ldots,2n+1\}$ & \cite{BFL:orth} \\ \hline
$\PSL(4,\F)$ & $q$ odd & $\supseteq\{3,4\}$ & \cite{BL:PSL4} \\ \hline
$\PSp(4,\F)$ & $q\neq 3$ & $\{3,4,5\}$ & Theorem~\ref{thm:psp4-main} \\ \hline\hline
\end{tabular}
\caption{Infinite families of simple groups (and their variants) known to have string C-group representations of rank at least 4.}
\end{figure}

To prove Theorem~\ref{thm:psp4-main} we use the Klein correspondence to move from the 4-dimensional 
projective representation of $\PSp(4,\F)$ to the 5-dimensional linear representation of its isomorphic
orthogonal group $\Omega(5,\F)$. This was also the approach in~\cite{BFL:orth}, where it was shown that
$\PSp(2n,\Fe)\cong \Omega(2n+1,\Fe)$ has string C-group representations of rank $2n+1$. (Note, 
in odd characteristic this 
isomorphism holds only when $n=2$.) That construction may be adapted in odd characteristic to generate
$\Orth(5,\F)$ as a string C-group of rank 5, but we must work harder to obtain such a representation for its 
simple subgroup $\Omega(5,\F)$
of index 4. Having done so, we apply a recently discovered technique~\cite{BL:rank-reduction} to reduce the rank of this 
representation and obtain the desired constructions of ranks 3 and 4.

\section{Orthogonal groups and their geometries}
\label{sec:classical}
Let $\F$ be the field wth $q=p^e$ elements, and $V$ a $d$-dimensional
$\F$-space. Although we shall be primarily interested in the case $d=5$, we work for a while
with arbitrary $d$. Let $\phi\colon V\to \F$ be a quadratic form on $V$, and
\begin{align}
\label{eq:define-bform}
\forall u,v\in V, & & (u,v)=\phi(u+v)-\phi(u)-\phi(v)
\end{align}
its associated symmetric bilinear form. 
If $q$ is odd, $\phi$ can be recovered from $(\,,\,)$ using the equation $\phi(v)=(v,v)/2$.

To each $U\leq V$ we associate a subspace $U^{\perp}=\{v\in V\colon (v,U)=0\}$, and say $V^{\perp}$ is the {\em radical} of $V$. If $d$ is even, or
$d$ and $|\F|$ are both odd, we insist that $V^{\perp}=0$. If $d$ is odd
and $|\F|$ is even, $V^{\perp}$ is always nonzero; here, we insist
 that $V^{\perp}=\la z\ra$
with $\phi(z)\neq 0$. A subspace $U$ of $V$
is {\em nonsingular} if the restriction of $\phi$ to $U$ has these properties.
In particular, a 1-space $\la v\ra$ is {\em nonsingular} if $\phi(v)\neq 0$,
and is otherwise {\em singular}
(we extend this terminology to vectors).

The {\em orthogonal group} corresponding to $\phi$ is the group
\begin{align}
\label{eq:orth-group}
\Orth(V) &= \{ g\in \GL(V) \colon \phi(vg)=\phi(v)~\mbox{for all}\;v\in V \}
\end{align}
of {\em isometries} of $\phi$.
Although we generally work with orthogonal groups and their 
underlying $\F$-spaces using standard
(coordinate-free) linear transformation notation, 
it will from time to time be helpful to compute explicitly with matrices.
If $x$ is a matrix, $x^{{\rm tr}}$ denotes its transpose.
Fixing an ordered basis $v_1,\ldots,v_d$ for $V$, we
represent $\phi$ as a $d\times d$ upper-triangular matrix $\Phi:=[[\phi_{ij}]]$, where
$\phi_{ii}=\phi(v_i)$, and $\phi_{ij}=(v_i,v_j)$ for $1\leq i<j\leq d$. Then $\Phi+\Phi^{{\rm tr}}$
is the matrix $[[(v_i,v_j)]]$ representing $(\,,\,)$ relative to $v_1,\ldots,v_d$, and
\begin{align}
\label{eq:orth-group-matrix}
\Orth(V) &= \{ g\in \GL(d,\F) \colon g\Phi g^{{\rm tr}}=\Phi \}.
\end{align}
Throughout this section we shall appeal 
to standard results from Taylor's text~\cite{Ta:geom_classical}, to which we refer the reader
for a thorough treatment of classical groups and their geometries.

\subsection{Symmetries, and their geometric properties.}
\label{subsec:involutions}
The involutions best suited to generating string C-groups of high rank are those
with $\pm 1$--eigenspaces of highest possible dimension. Such involutions arise from elements
of $\Orth(V)$ known as {\em symmetries} that are defined in terms of a nonsingular vector 
$u\in V$ as follows:
\begin{align}
\label{eq:symmetry}
\forall v\in V, &&
\sigma_u\colon v\mapsto v-\frac{(u,v)}{\phi(u)}\,u.
\end{align}
Observe that $\sigma_u$ is the identity on the $(d-1)$-space space $u^{\perp}$. 
If $|\F|$ is odd, $\sigma_u(u)=-u$,
so $\sigma_u$ is a {\em reflection} and has determinant $-1$. 
If $|\F|$ is even, $\sigma_u$ is also
the identity on $V/\la u\ra$,
and hence is a {\em transvection}. 
Note, $\sigma_u=\sigma_{\lambda u}$ for all
$\lambda\in\F^{\times}$, so symmetries correspond to nonsingular points of the projective geometry
$\Bbb{P}(V)$ and we write $\sigma_u=\sigma_{\la u\ra}$. 
For $X\subseteq V$, define
\begin{align}
\label{eq:define-SigmaX}
\Sigma(X) &= \la \sigma_u\colon u\in X \ra.
\end{align}

We typically frame our arguments in geometric
terms, and the following elementary observation is particularly useful. For nonsingular points
$x,y$ in $\Bbb{P}(V)$,

\begin{align}
\label{eq:sig-comm}
[\sigma_x,\sigma_y]=1 & ~\Longleftrightarrow~\;\mbox{$y$ lies on $x^{\perp}$.} 
\end{align}
The {\em support} of $h\in\GL(V)$ is the subspace $[V,h]:=\{v - vh\colon v\in V\}$. If
$H\leq \GL(V)$ is generated by $S$, then its support is $[V,H]:=\sum_{h\in S}[V,h]$.

\subsection{Generating with symmetries}
\label{subsec:gen-inv}
It is well known that unless $V=\Bbb{F}_2^4$ the orthogonal group 
$\Orth(V)$ is generated by its symmetries; cf.~\cite[Theorem 11.39]{Ta:geom_classical}.
If $\dim V$ is odd and $|\F|$ is even---often considered the degenerate case---then $\Orth(V)\cong \Sp(V/V^{\perp})$ is the
simple group we wish to work with. In the other cases $\Orth(V)$ has normal subgroups,
which we now describe. 

The map $g\mapsto \dim [V,g]\;({\rm mod}\;2)$ is a
group homomorphism $\pi\colon \Orth(V)\to \Bbb{Z}_2$ having kernel of index 2.
In fact, if $g=\sigma_{u_1}\ldots\sigma_{u_r}$, then $\dim [V,g]\;({\rm mod}\;2)=r\;({\rm mod}\;2)$
and $\ker\pi$ is the subgroup of $\Orth(V)$ consisting of products of even numbers of 
symmetries. Next, the map
$g\mapsto \phi(u_1)\ldots\phi(u_r)(\F^{\times})^2$ is a (well-defined)
homomorphism $\theta\colon\Orth(V)\to\F^{\times}/(\F^{\times})^2$,
and $\theta(g)$ is called the {\em spinor norm} of $g$. 
For convenience, we denote the image of $\theta$ by $\{\bar{0},\bar{1}\}$, where
$\bar{0}$ is the square class, and $\bar{1}$ is the non-square class.
Evidently, $\ker\theta$ has
index 2 in $\Orth(V)$ if, and only if, $|\F|$ is odd, in which case 
\begin{align}
\label{eq:define-ker-theta}
\ker\theta & = \la \sigma_{u}\colon u\in V~\mbox{is nonsingular, and}~\phi(u)~\mbox{is a square in}\;\F^{\times}\ra.
\end{align}
Our principal focus is the subgroup
\begin{align}
\label{eq:define-omega}
\Omega(V) &= \left\{ \begin{array}{cl} 
\ker\pi\cap\ker\theta & \mbox{if}~|\F|~\mbox{is odd} \\
\Orth(V) & \mbox{if}~|\F|~\mbox{is even, and}~\dim(V)~\mbox{is odd} \\
\ker \pi &  \mbox{if}~|\F|~\mbox{is even, and}~\dim(V)~\mbox{is even} 
\end{array} \right.
\end{align}
of $\Orth(V)$. 
{\em We restrict now to the case when $\dim V$ is odd}, and consider generation of $\Omega(V)$
using symmetries. 
\medskip

When $|\F|$ is even, $\Orth(V)=\Omega(V)$ is generated by symmetries. 
\medskip

When $|\F|$ is odd, $\ker\pi=\SO(V)$, the subgroup of determinant 1
isometries of $\phi$. Consider the map $\Orth(V)\to \Bbb{Z}_2\times \{\bar{0},\bar{1}\}$
sending $g\mapsto (g\pi,g\theta)$. The proper subgroups of the codomain give three maximal
subgroups of $\Orth(G)$: $M_1=\ker\theta$ is the preimage of $\la (1,\bar{0})\ra$,
$M_2$ is the preimage of $\la (1,\bar{1})\ra$, and $\SO(V)=\ker \pi$ is the preimage
of $\la (0,\bar{1})\ra$.

Consider $-1\in\Orth(V)$. As $\det(-1)=-1$ and $M_1\cap M_2=\Omega(V)\leq\SO(V)$, it follows that $-1$ 
belongs to precisely one of the maximal subgroups $M_i$. As $\det(-\sigma_u)=(-1)(-1)=1$ for each nonsingular 
$u\in V$, it follows that 
\begin{align*}
\la -\sigma_u\colon \phi(u)~\mbox{is a square} \ra, & & \la -\sigma_u\colon \phi(u)~\mbox{is a non-square} \ra
\end{align*}
are both subgroups of $\SO(V)$. Indeed, since $\dim V$ is odd, 
\begin{align}
\label{eq:omega-gen}
\Omega(V) &= \la -\sigma_u\colon \theta(-1)=\phi(u)\,(\F^{\times})^2 \ra.
\end{align}
Of course, one can simply multiply $\phi$ by a suitable scalar to ensure that $\theta(-1)=\bar{0}$, but we 
do not wish to place constraints on the specific $\phi$ that we work with.

For $X\subseteq V$, define subsets
\begin{align}
\label{eq:X-square}
X_{\square}=\la  u\in X\colon \phi(u)~\mbox{is a square}\ra &&
X_{\boxtimes}=\la u\in X\colon \phi(u)~\mbox{is a non-square}\ra
\end{align}
of $X$. We record the following result for easy reference.

\begin{lemma}
\label{fact:generation}
If $(V,\phi)$ is a 5-dimensional orthogonal $\F$-space, then either
\smallskip

(i) $\theta(-1)=\bar{0}$ and $\PSp(4,\F)\cong \Omega(V)=\la -\sigma_u\colon u\in V_{\square} \ra$, or
\smallskip

(ii) $\theta(-1)=\bar{1}$ and $\PSp(4,\F)\cong \Omega(V)=\la -\sigma_u\colon u\in V_{\boxtimes} \ra$.
\end{lemma}

\begin{remark}
\label{rem:d-odd-trick}
This trick for generating $\Omega(V)$ with elements having an eigenspace of dimension $\dim V-1$
works whenever $\dim V$ is odd. Indeed, by extending the methods in this paper
it is anticipated that {\em for all finite fields $\F$} of order at least 4, 
the simple group $\Omega(2n+1,\F)$ has a string C-group representation of rank $2n+1$.
When $\dim V=2n$ is even, on the other hand, $-\sigma_u$ still has determinant $-1$.
Thus, it seems unlikely that $\Omega^{\pm}(2n,\F)$ has string C-group representations of
rank $2n$.
\end{remark}

The following version of~\cite[Proposition 3.3]{BFL:orth} is more restrictive in that it applies just to
nonsingular subspaces of an orthogonal space $V$, but it also encompasses all finite fields.
The restricted version is all we need here, and the proof is simpler than that of~\cite[Proposition 3.3]{BFL:orth}.

\begin{prop}
\label{prop:intersections}
Let $U,W$ be nonsingular subspaces of an orthogonal $\F$-space $V$ such that $U\cap W$
is nonsingular. Then $\Sigma(U)\cap\Sigma(W)=\Sigma(U\cap W)$.
\end{prop}

\begin{proof}
Let $G=\Orth(V)$. For $X\subseteq V$, 
${\rm cent}_G(X)=\{g\in G\colon \forall x\in X,~xg=x\}$
is the subgroup of $G$ fixing $X$ pointwise.
As $U$ is a nonsingular subspace of $V$, the stabilizer ${\rm stab}_G(U)=\{g\in G\colon Ug = U\}$
factorizes as a direct product 
\begin{align*}
{\rm stab}_G(U)&={\rm cent}_G(U)\times{\rm cent}_G(U^{\perp}).
\end{align*}
Furthermore, ${\rm cent}_G(U^{\perp})$ induces the full group $\Orth(U)$ of
isometries on $U$; we say ${\rm cent}_G(U^{\perp})$ is a {\em natural embedding} 
of $\Orth(U)$ in $G$. In particular, ${\rm cent}_G(U^{\perp})=\Sigma(U)$.

As $W$ and $U\cap W$ are also
nonsingular, we have $\Sigma(U\cap W)\leq \Sigma(U)\cap \Sigma(W)$. 
If $g\in \Sigma(U)\cap\Sigma(W)={\rm cent}_G(U^{\perp})
\cap{\rm cent}_G(W^{\perp})$, then $g$ is the identity on $U^{\perp}$ and $W^{\perp}$ and hence
on $U^{\perp}+W^{\perp}=(U\cap W)^{\perp}$. As 
 $U\cap W$ is nonsingular, it follows that 
 ${\rm cent}_G((U\cap W)^{\perp})=\Sigma(U\cap W)$,
 so the result follows.
\end{proof}

\begin{remark}
\label{rem:restrict}
Restricting the types of symmetries used to generate---working instead with the groups
$\Sigma(X_{\square})$ and $\Sigma(X_{\boxtimes})$---yields equivalent results, namely 
\begin{align*}
\Sigma(U_{\square})\cap \Sigma(W_{\square})=\Sigma((U\cap W)_{\square}),~\mbox{and}
&&
\Sigma(U_{\boxtimes})\cap \Sigma(W_{\boxtimes})=\Sigma((U\cap W)_{\boxtimes}).
\end{align*}
\end{remark}

\section{The rank 5 construction}
\label{sec:rank5}
The construction of a string C-group representation for
$\PSp(4,\F)\cong\Omega(5,\F)$ is similar
to the one given in~\cite{BFL:orth} to prove that 
$\Orth(2m+1,\Fe)\cong\PSp(2m,\Fe)$ 
is a string C-group of rank $2m+1$. 
We could of course appeal to~\cite{BFL:orth} 
for the case $\Fe$ and focus just on $|\F|$ odd, but
we prefer to give a (somewhat uniform) self-contained treatment for all finite fields.
Before digging in, though, we record the following useful shortcut to  
string C-group verification.

 \begin{lemma}
 \label{lem:ip}
 ~Let $G$ be a group generated by involutions
$\rho_0,\ldots,\rho_{n-1}$ such that
   \begin{enumerate}
   \item[(i)] $\la \rho_0,\ldots,\rho_{n-2}\ra$ and $\la \rho_1,\ldots,\rho_{n-1}\ra$ are both string C-groups, and
   \item[(ii)] $\la \rho_0,\ldots,\rho_{n-2}\ra\cap \la \rho_1,\ldots,\rho_{n-1}\ra = \la \rho_1,\ldots,\rho_{n-2}\ra$.
   \end{enumerate}
   Then $(G\,;\,\{ \rho_0,\ldots,\rho_{n-1}\})$ is a string C-group representation.
\end{lemma}

\begin{proof}
See~\cite[Proposition 2E16]{McS:ARP}.
\end{proof}

Let $(V,\phi)$ be an orthogonal $\F$-space, and
$u_1,\ldots,u_d$ a basis of $V$ consisting of nonsingular vectors.
Consider a putative string C-group representation of a group
generated by symmetries $\sigma_{u_1},\ldots,\sigma_{u_d}$. 
From the commutator relation~\eqref{eq:sig-comm}, we see that the string 
condition~\eqref{eq:sp} 
translates to an orthogonality condition on their associated vectors $u_1,\ldots,u_d$, namely
$(u_i,u_{i+k})=0$ if, and only if, $k>1$. 
One can replace $u_i$ with $u_i/(u_i,u_{i+1})$ for $i<d$ to ensure that
$(u_i,u_{i+1})=1$ without changing the symmetries. 
Thus, $\phi$ is represented relative to $u_1,\ldots,u_d$ by a matrix 
\begin{align}
\label{eq:define-F-alpha}
\Phi(\alpha_1,\ldots,\alpha_d) &=
\left[ \begin{array}{ccccc} \alpha_1 & 1 & & & \\  & \alpha_2 & 1 & & \\ &  & \ddots & \ddots & \\ 
& &  & \alpha_{d-1} & 1 \\ & & &  & \alpha_d
\end{array} \right].
\end{align}
Restricting to the case $d=5$, our strategy is to choose the scalars $\alpha_i$ so that everything 
works nicely. For us, this will mean that the following conditions hold.
\begin{enumerate}
\item[(a)] For each $1\leq i\leq 5$,  the dihedral group $\la \sigma_{u_i},\sigma_{u_{i+1}}\ra$
is equal to $\Sigma(\la u_i,u_{i+1}\ra)$ if $|\F|$ is even, or to one of $\Sigma(\la u_i,u_{i+1}\ra_{\square})$
or $\Sigma(\la u_i,u_{i+1}\ra_{\boxtimes})$ if $|\F|$ is odd. This 
means that each product $\sigma_{u_i}\sigma_{u_{i+1}}$
has order $q\pm 1$ if $q$ is even, or $(q\pm 1)/2$ if $q$ is odd.
An easy calculation shows that the restriction of this product to $\la u_i,u_{i+1}\ra$
is represented by the matrix
\begin{align*}
h(\alpha_i,\alpha_{i+1}) &= \left[ \begin{array}{cc} -1 & 1/\alpha_{i+1} \\ -1/\alpha_i & 1/(\alpha_i\alpha_{i+1})-1 \end{array} \right],
\end{align*}
\item[(b)] For each $1\leq i<j\leq 5$, the subspace $U_{ij}=\la u_i,\ldots u_j\ra$ is nonsingular.
\item[(c)] If $q$ is odd, then $\theta(\sigma_{u_i})=\theta(-1)$ for each $1\leq i\leq 5$.
\end{enumerate}
To achieve this, we consider separately the cases $|\F|$ even and $|\F|$ odd.
\medskip

\noindent {\bf \boldmath$|\F|$ is even.}~ 
Fix $\xi\in\F^{\times}$ such that $h(\xi,\xi)$ has order $q\pm 1$. For $\alpha\in\F^{\times}-\{\xi\}$, consider
the quadratic form $\phi$ represented by the matrix
\begin{align}
\label{eq:form-char2}
\Phi_{{\rm even}}(\alpha):=\Phi(\xi,\xi,\alpha,\alpha,\xi).
\end{align}
For condition (a), we simply require that both $h(\alpha,\alpha)$ and $h(\xi,\alpha)$ have order $q\pm 1$.
For condition (b), 
notice that 
\begin{align*}
U_{ij}^{\perp}\cap U_{ij} &= \left\{ \begin{array}{cl}
0 & \mbox{if $j-i$ is odd} \\
\la u_i+u_{i+2}\ra & \mbox{if $j=i+2$, and} \\
\la u_1+u_3+u_5\ra & \mbox{if $i=1$ and $j=5$}
\end{array} \right.
\end{align*}
As $\phi(u_i+u_{i+2})=\xi+\alpha\neq 0$ for $1\leq i\leq 3$, and $\phi(u_1+u_3+u_5)=\alpha\neq 0$,
it follows that $U_{ij}$ is always nonsingular, as required. Thus, we define
 \begin{align}
 \label{eq:define-Gamma-even}
 \Gamma_{{\rm even}}(\xi) &:= \{ \alpha\in\F^{\times}-\{\xi\} \colon~\mbox{both}~
 h(\alpha,\alpha)~\mbox{and}~h(\xi,\alpha)~\mbox{have order}\;q\pm 1\},
 \end{align}
 so that $\Phi_{{\rm even}}(\alpha)$ satisfies conditions (a) and (b)
 for any $\alpha\in\Gamma_{{\rm even}}(\xi)$.
 \medskip
 
 \noindent {\bf \boldmath$|\F|$ is odd.}~ 
 Fix $\xi\in\F^{\times}-\{\frac{1}{2}\}$ such that $\xi^2-\frac{1}{2}$ is a nonzero square, 
 and $h(1,\xi^2)$ has order $(q\pm 1)/2$. For $\alpha\in\F^{\times}$
 consider the quadratic form $\phi$ represented by 
\begin{align}
\label{eq:form-char-odd}
\Phi_{{\rm odd}}(\alpha):=\Phi(1,\xi^2,1,\alpha^2,1).
\end{align}
Condition (a) is satisfied so long as $h(1,\alpha^2)$ has order $(q\pm 1)/2$,
so let $A$ denote the set of squares in $\F^{\times}$ having this property. Choosing
the diagonal entries of $\Phi$ to be squares ensures that all $\theta(\sigma_{u_i})=\bar{0}$.
Hence, for $1\leq i\leq 4$ we have $\la \sigma_{u_i},\sigma_{u_{i+1}} \ra=\Sigma(\la u_i,u_{i+1} \ra_{\square})$.

Computing determinants of submatrices of the matrix
$\Phi_{{\rm odd}}(\alpha)+\Phi_{{\rm odd}}(\alpha)^{{\rm tr}}$
representing the symmetric bilinear form associated to $\phi$, 
we observe that
condition (b) holds provided $\alpha^2$ is selected from the set
\begin{align}
\label{eq:Gamma0}
\Gamma_0(\xi) &= 
A-
\left\{ \frac{1}{2},\,\frac{1}{4},\, \frac{\xi^2}{4\xi^2-1},\,\frac{2\xi^2-\frac{1}{4}}{4\xi^2-1},\, \frac{\xi^2-\frac{1}{4}}{4\xi^2-2},\,
\frac{\xi^2-\frac{3}{8}}{2\xi^2-1} 
\right\} 
\end{align}

To analyze condition (c), we compute an orthogonal basis $(w_i\colon 1\leq i\leq 5)$ for $V$. For then,
$-1=\prod_i\sigma_{w_i}$, so that $\theta(-1)=\prod_i\phi(w_i)\,(\F^{\times})^2$. Put
$U=\la u_1,u_3,u_5\ra$, so that $U^{\perp}=\la -u_1+2u_2-u_3,-u_3+2u_4-u_5\ra$.
As
\begin{align*}
\phi(-u_1+2u_2-u_3) = 4\xi^2-2 & & \phi(-u_3+2u_4-u_5)=4\alpha^2-2,
\end{align*}
and $\xi^2,\alpha^2\neq\frac{1}{2}$, it follows that $w_1:=-u_1+2u_2-u_3$ and $-u_3+2u_4-u_5$
are nonsingular.
Notice $\phi(w_1) = 4\xi^2-2=2^2(\xi^2-\frac{1}{2})$ is a square, so $\theta(\sigma_{w_1})=\bar{0}$.
 Also, 
\begin{align*}
w_2 &:=-u_1+2u_2+(1-4\xi^2)u_3+(8\xi^2-4)u_4+(2-4\xi^2)u_5
        \in w_1^{\perp}\cap U^{\perp},
\end{align*}
so we compute
\begin{align*}
\phi(w_2) &=16\alpha^2(2\xi^2-1)^2-(32\xi^4-28\xi^2+6).
\end{align*}
Define $m(\xi)=16(2\xi^2-1)^2\neq 0$, $b(\xi)=-(32\xi^4-28\xi^2+6)$, and
\begin{align}
\label{eq:define-Gamma-odd}
\Gamma_{{\rm odd}}(\xi) &= \{ \alpha^2\in \Gamma_0(\xi)\colon \alpha^2m(\xi)+b(\xi)~\mbox{is a nonzero square} \}.
\end{align}
We have shown that $\Phi_{{\rm odd}}(\alpha)$ satisfies conditions (a), (b) and (c) for any $\alpha\in \Gamma_{{\rm odd}}(\xi)$.
Note, we could have chosen $\xi$ so that $\xi^2-\frac{1}{2}$ is a non-square, in which case the condition on $\alpha$
in~\eqref{eq:define-Gamma-odd} would change to $\alpha^2m(\xi)+b(\xi)$ a non-square. This increases the pool
of scalars for a fixed $\F$, but our choice suffices to establish existence.

\begin{prop}
\label{prop:rank5}
Let $q\geq 4$, $\F$ be the field with $q$ elements, and  
$\xi\in\F$ any element having the properties described separately above for $q$ even and $q$ odd.  Let
$\Gamma(\xi)$ be the set defined in~\eqref{eq:define-Gamma-even} or~\eqref{eq:define-Gamma-odd}
for $q$ even or odd, respectively. Suppose $\Gamma(\xi)\neq\emptyset$, and fix $\alpha\in\Gamma(\xi)$. 
If $\phi$ is the quadratic form represented, relative to basis $(u_i\colon 1\leq i\leq 5)$, by the matrix 
$\Phi_{{\rm even}}(\alpha)$ or $\Phi_{{\rm odd}}(\alpha)$, then
\begin{align*}
(\;\Omega(V;\phi)\;;~\{-\sigma_{u_i}\colon 1\leq i\leq 5\}\;)
\end{align*}
is a string C-group representation of $\Omega(V)\cong\PSp(4,\F)$
of Schl\"{a}fli type $[p,p,p,p]$, where $p\in \{q-1,q+1\}$ if $q$ is even, and $p\in\{(q-1)/2,(q+1)/2\}$
if $q$ is odd.
\end{prop}

\begin{proof}
We first claim that, for any such choice of scalars $\xi,\alpha$:
\begin{align}
\label{eq:gen}
\mbox{for}\;1\leq i<j\leq 5, & & \la \sigma_{u_i},\ldots,\sigma_{u_j}\ra= \left\{
\begin{array}{ll}
\Sigma(\la u_i,\ldots,u_j\ra) & \mbox{if $q$ is even} \\
\Sigma(\la u_i,\ldots,u_j\ra_{\square}) & \mbox{if $q$ is odd} 
\end{array}
\right..
\end{align}
Recall that for $\alpha\in\Gamma(\xi)$, conditions (a), (b) and (c) are satisfied for the form
$\phi$ represented by matrix $\Phi_{{\rm even}}(\alpha)$ or $\Phi_{{\rm odd}}(\alpha)$.
In particular, condition (a) ensures that~\eqref{eq:gen} holds whenever $j=i+1$. This is the base
case of an induction on $j-i$.
For $q$ even, follow the argument in~\cite[Lemma 5.3]{BFL:orth}---although the form matrix 
is different, only conditions (a) and (b) matter.
The $q$ odd case is identical except
that $\la \sigma_{u_i},\ldots,\sigma_{u_j}\ra=\Sigma(\la u_i,\ldots,u_j\ra_{\square})$
since we generate only with symmetries having
square spinor norm; see Remark~\ref{rem:restrict}.

We next show, again using induction, that $\sigma_{u_1},\ldots,\sigma_{u_5}$ generates a string C-group.
For $1\leq i\leq 4$, the dihedral group $\la \sigma_{u_i},\sigma_{u_{i+1}} \ra$ (with its defining generators) 
is clearly a string C-group. For $j-i>1$, consider the group $\la \sigma_{u_i},\ldots,\sigma_{u_j}\ra$. By
induction, both $\la \sigma_{u_i},\ldots,\sigma_{u_{j-1}}\ra$ and $\la \sigma_{u_{i+1}},\ldots,\sigma_{u_j}\ra$
are string C-groups. Furthermore, by Proposition~\ref{prop:intersections} and the claim above,
\begin{align*}
\la \sigma_{u_i},\ldots,\sigma_{u_{j-1}}\ra \cap \la \sigma_{u_{i+1}},\ldots,\sigma_{u_j}\ra & = \Sigma(\la u_i,\ldots,u_{j-1}\ra)\cap\Sigma(\la u_{i+1},\ldots,u_j\ra) \\
        & = \Sigma(\la u_i,\ldots,u_{j-1}\ra \cap \la u_{i+1},\ldots,u_j\ra) \\
        & = \Sigma(\la u_{i+1},\ldots,u_{j-1}\ra) \\
        & = \la \sigma_{u_{i+1}},\ldots,\sigma_{u_{j-1}} \ra.
\end{align*}
It now follows from Lemma~\ref{lem:ip} that $\la \sigma_{u_i},\ldots,\sigma_{u_j} \ra$ is a string C-group.

Finally, using the initial claim again, together with the discussion in Section~\ref{subsec:gen-inv}, we have
$\la \sigma_{u_1},\ldots,\sigma_{u_5} \ra=\ker\theta$. When $q$ is even, $\Orth(V)=\ker\theta=\Omega(V)$, so
$(\Omega(V)\,;\,\{\sigma_{u_1},\ldots,\sigma_{u_5}\})$
is a string C-group representation.
When $q$ is odd, the choices of $\xi$ and $\alpha$ ensure that condition (c) holds, so that $\theta(-1)=\bar{0}$. Thus,
by Lemma~\ref{fact:generation}(i), $(\Omega(V)\,;\,\{-\sigma_{u_1},\ldots,-\sigma_{u_5}\})$
is a string C-group representation.
\end{proof}

\section{Proof of Theorem~\ref{thm:psp4-main}}
\label{sec:proof}
We first limit the rank of a string C-group representation for $G=\PSp(4,\F)$. Once again we approach
this within the context of the isomorphic group $\Omega(V)$, where $V$ is a 5-dimensional orthogonal space
equipped with quadratic form $\phi$.

\begin{lemma}
\label{lem:rank-restrict}
Let $\F$ be any finite field, $V$ a 5-dimension $\F$ space equipped with quadratic form $\phi$,
and $n\geq 6$ an integer.
If $H$ is a subgroup of $\Orth(V)$ generated 
by a sequence of involutions $\rho_0,\ldots,\rho_{n-1}$ 
satisfying~\eqref{eq:sp}, 
then $\rho_0,\ldots,\rho_{n-1}$ violates the intersection condition~\eqref{eq:ip}.
\end{lemma}

\begin{proof}
 A non-central involution of $\Orth(V)$
is one of $\pm\sigma_u$ or $\pm \sigma_u\sigma_w$ for $u,w\in V$ nonsingular and $w\in u^{\perp}$.
Consider the non-commuting subgroups $L=\la \rho_0,\rho_1,\rho_2\ra$ and 
$R=\la \rho_3,\rho_4,\rho_5\ra$ of $H$ (possible since $n\geq 6$). 
As each product $\rho_i\rho_{i+1}$ has order exceeding 2, $L$  induces on 
$[V,L]$, of dimension at least 3, a group generated by non-commuting dihedral groups.
Similarly $R$ induces on $[V,R]$, of dimension at least 3, a second such group. 
As $\dim V=5$, so $[V,L]\cap[V,R]$ is nontrivial. It follows that $L\cap R$
is also nontrivial, so condition~\eqref{eq:ip} fails for the sequence $\rho_0,\ldots,\rho_{n-1}$.
\end{proof}

Lemma~\ref{lem:rank-restrict} shows that $\rk(G)\subseteq \{3,4,5\}$.
To complete the proof of Theorem~\ref{thm:psp4-main} we establish the
existence of a suitable rank 5 string C-group representation for 
$G\cong\Omega(V)$. 
This will show that $5\in\rk(G)$. We will
then apply the following {\em rank reduction} technique to show that $\{3,4\}\subseteq \rk(G)$.

\begin{theorem}
\label{thm:rank-reduction}
Let $(G;\{\rho_0,\ldots, \rho_{n-1}\})$ be an irreducible string C-group 
representation of rank $n\geq 4$.
If $\rho_2\rho_3$ has odd order, then 
$(G;\{\rho_1, \rho_0\rho_2, \rho_3, \ldots, \rho_{n-1}\})$ is a 
string C-group representation of rank $n-1$.
\end{theorem}

\begin{remark}
This `Petrie like' construction was developed in~\cite{HL:Petrie} and first used as a rank reduction technique
for the symmetric groups in~\cite{FL:symmetric}. The general technique, along with the criteria in Theorem~\ref{thm:rank-reduction}, 
was established in the recent paper~\cite{BL:rank-reduction}.
\end{remark}

\begin{proof}[Proof of Theorem~\ref{thm:psp4-main}]
As the entire result can be verified by brute force for moderate values of $q$ on a computer
(see Section~\ref{subsec:software}),
for convenience we assume that $q\geq 4$ if $q$ is even and that $q\geq 11$ if $q$ is odd.
Referring to Proposition~\ref{prop:rank5}, our first task is to establish the existence of a suitable
scalar $\xi$, and to show that $\Gamma(\xi)$ is nonempty.
We consider $q$ even and odd separately.

Referring to~\eqref{eq:define-Gamma-even} for $q$ even, the only condition is that there are distinct
scalars $\xi,\alpha\in\F^{\times}$ such that the matrices $h(\xi,\xi),\,h(\alpha,\alpha),\,h(\xi,\alpha)$ have 
order $q\pm 1$. For each $\xi,\alpha\in\F^{\times}-\{0,1\}$, all of these matrices
have order dividing $q\pm 1$. The proportion of scalars defining generators of those cyclic groups behaves 
as $1/(c\log\log q)$ and there are at least two such scalars in $\F^{\times}$ for all $q\geq 4$.

Referring to~\eqref{eq:Gamma0} and~\eqref{eq:define-Gamma-odd} for $q$ odd, while we
halve the number of scalars we consider by only using squares, the proportion of these elements
generating cyclic groups of the appropriate order is the same as in the even case.
Upon restricting to $\Gamma_0(\xi)$, we discard up to 6 more scalars.
Finally, consider the condition (as $\alpha^2$ ranges over $\Gamma_0(\xi)$ for fixed $\xi$) that $\alpha^2m(\xi)+b(\xi)$ 
is a square. If $b(\xi)=0$, the condition always holds, but
otherwise it will hold for roughly half of the choices of $\alpha^2$.
In particular, $\Gamma_{{\rm odd}}(q)$ is nonempty so long as $q$ is large enough
(one can check on the computer that indeed $q\geq 11$ is large enough).

Thus, by Proposition~\ref{prop:rank5}, for $q$ large enough there is a string C-group representation
of rank 5 for $\Omega(V)$. When $q$ is odd, moreover, it is possible to choose scalars so that
the Schl\"{a}fli type of the representation consists of any combination of $(q-1)/2$ and $(q+1)/2$. In particular,
observing the residue class of $q$ modulo 4, we can arrange for this to be a sequence of odd numbers.
(Note, when $q$ is even, our construction always produces Schl\"{a}fli types of odd numbers $q\pm 1$.)
Hence, we can now simply apply Theorem~\ref{thm:rank-reduction} twice to our string C-group to obtain
new representations of ranks 4 and 3. This completes the proof.
\end{proof}

\begin{remark}
\label{rem:rank3}
Although our proof gives string C-group representations of rank 3 for $\PSp(4,\F)$ via rank reduction,
we remark that at least one other construction was already known. Indeed, as mentioned in the introduction, a
complete determination of the simple groups of Lie type having string 
C-group representations of rank 3 may be extracted from the work of Nuzhin. 
\end{remark}

\section{Concluding remarks}

In this final section of the paper we make a number of remarks pertaining to the
results of the foregoing sections.

\subsection{Choosing scalars}
\label{subsec:scalars}
We chose scalars $\xi,\alpha\in\F^{\times}$ in Section~\ref{sec:rank5} and defined 
quadratic forms $\Phi_{{\rm even}}(\alpha)$ and $\Phi_{{\rm odd}}(\alpha)$
so as to make the verification of the corresponding string C-group representation
as direct as possible. However, there is a lot of flexibility to choose other scalars
in such a way that new (non-isomorphic) representations are obtained. For instance
we do not absolutely require that all of the subspaces $\la u_i,\ldots,u_j\ra$ be 
nonsingular;
indeed, the construction in~\cite{BFL:orth} does not insist on this. 

\subsection{Higher rank representations in higher dimensions} 
\label{subsec:higher-rank}
As we have already noted, the idea in~\cite{BFL:orth} is to use symmetries to represent
$\Orth(d,\Fe)$ as a string C-group of rank $d$. It was not fully appreciated at the time that
the same technique can be used to build rank $d$ representations for $\Orth(d,\F)$
when $|\F|$ is odd. By restricting the scalars $\alpha_i$ in~\eqref{eq:define-F-alpha} to be either all
squares or all non-squares, one can further generate the two maximal subgroups of
$\Orth(d,\F)$ generated by symmetries. It seems likely, then, that our trick for generating
$\Omega(5,\F)$ by negating symmetries extends to $\Omega(2n+1,\F)$,
so that $2n+1\in\rk(\Omega(2n+1,\F))$ for all $q\geq 4$. However, we have no intuition 
to offer in regard to the highest rank string C-group representation of the quasisimple
groups $\Omega^{\pm}(2n,\F)$.

\subsection{Another rank 4 construction} 
\label{subsec:rank4}
In an earlier draft of this paper---before the emergence of the rank reduction technique---a direct 
construction of a string C-group
representation of rank 4 for $\Omega(V)$ was discovered that utilized, to a greater extent, geometric properties
of involutions. For completeness we provide this construction, but in the interest of brevity 
omit the verification.

Let $q\geq 5$ be an odd prime power, and $\F$ the field of $q$ elements; an explicit rank 4 construction of $\Orth(5,\Fe)$
was given in~\cite[Theorem 6.1]{BFL:orth}. The tuple
\[
  \left[ \begin{smallmatrix} -1 & \cdot & \cdot & \cdot & \cdot \\ \cdot & 1 & \cdot & \cdot & \cdot \\ \cdot & \cdot & 1 \cdot & \cdot \\ \cdot & \cdot & \cdot & 1 & \cdot \\ \cdot & \cdot & \cdot & \cdot & -1 \end{smallmatrix} \right],~~  
   \left[ \begin{smallmatrix}  f(\alpha) & g(\alpha) & \cdot & \cdot & \cdot  \\
                                                      g(\alpha) & -f(\alpha) & \cdot & \cdot &  \cdot \\
                                                       \cdot & \cdot & 1 & \cdot & \cdot \\ 
                                                     \cdot & \cdot & \cdot & -f(\alpha) & g(\alpha) \\
                                                    \cdot & \cdot  & \cdot & g(\alpha) & f(\alpha)
                    \end{smallmatrix} \right],~~              
    \left[ \begin{smallmatrix} -1 & \cdot & \cdot & \cdot & \cdot \\
                                                 \cdot & f(\beta) & -g(\beta) & \cdot & \cdot \\
                                                    \cdot & -g(\beta) & -f(\beta) & \cdot  & \cdot \\
                                                    \cdot & \cdot & \cdot & -1 & \cdot \\
                                                     \cdot & \cdot & \cdot & \cdot & -1
                    \end{smallmatrix}
            \right],~~         
    \left[ \begin{smallmatrix} \cdot & \cdot & \cdot & \cdot  & 1 \\
                                                 \cdot & \cdot & \cdot & 1 & \cdot \\
                                                    \cdot & \cdot & 1 & \cdot & \cdot \\
                                                    \cdot & 1 & \cdot & \cdot & \cdot \\
                                                    1 & \cdot & \cdot & \cdot & \cdot
                    \end{smallmatrix}
            \right],
 \]
  where $\alpha$ and $\beta$ come from (different) restricted sets of scalars, and 
\begin{align}
\label{eq:define-fg}
         f(\alpha) &=  \frac{1-\alpha^2}{1+\alpha^2} &
                                                   g(\alpha) &= \frac{-2\alpha}{1+\alpha^2},
\end{align}
defines a string C-group representation for $\Omega(5,\F)\cong\PSp(4,\F)$ of rank 4. It is not
isomorphic to the one given in the proof of Theorem~\ref{thm:psp4-main}, which applied reduction to
the rank 5 representation: the first, second, and fourth generators all have
2-dimensional $-1$-eigenspace, whereas the earlier construction had three generators with
4-dimensional $-1$-eigenspace.

\subsection{Software}
\label{subsec:software}
A software package called SGGI 
has been implemented by the author in the {\sc Magma} system~\cite{magma} and is publicly
available on GitHub. Among other things the package contains 
the explicit constructions in this paper, including the one in Section~\ref{subsec:rank4}, and functions 
to carry out exhaustive searches for string C-group
representations. These functions can be used to verify Theorem~\ref{thm:psp4-main} for small values of $q$, but
we caution that  $\PSp(4,\Bbb{F}_{19})$ is nearing the limit of practicability for
exhaustive searches. We note that other functions to compute with string C-groups, written by 
Dimitri Leemans, are distributed with {\sc Magma}.
\bigskip

\noindent {\bf Acknowledgments.} The author thanks J.T. Ferrara for the collection and synthesis of experimental data
that was particularly helpful when designing the rank 4 construction in Section~\ref{subsec:rank4}.
He also thanks the anonymous referees for several helpful suggestions.

\bibliographystyle{amsalpha}

\begin{thebibliography}{A}

\bibitem[BCP]{magma} W. Bosma, J. Cannon, C. Playoust,
{\em The Magma algebra system. I. The user language.},
Computational algebra and number theory (London 1993), J. Symbolic Comput. 24 (1997), no. 3-4, 235--265.

\bibitem[BFL]{BFL:orth}
P.A. Brooksbank, J.T. Ferrara, D. Leemans, {\em Orthogonal groups in characteristic 2 
acting on polytopes of high rank}, Discrete Comput. Geom., to appear.
{\tt https://arxiv.org/abs/1709.02219}.

\bibitem[BL1]{BL:PSL4}
P.A. Brooksbank, D. Leemans, {\em Polytopes of large rank for $\PSL(4,\F)$}, J. Algebra, 452 (2016), 390--400.

\bibitem[BL2]{BL:rank-reduction}
P.A. Brooksbank, D. Leemans, 
{\em Rank reduction of string C-group representations},
Proc. Amer. Math. Soc., to appear. {\tt https://arxiv.org/abs/1812.01055}.

\bibitem[BV]{BV:PSL3}
P.A. Brooksbank, D.A. Vicinsky, {\em Three-dimensional classical groups acting on polytopes},
Discrete Comput. Geom. 44, no. 3 (2010), 654--659.

\bibitem[Co1]{Co:alt1}
M. Conder, {\em Generators for the alternating and symmetric groups}, 
J. London Math. Soc. (2), 22 (1980), 75--86.

\bibitem[Co2]{Co:alt2}
M. Conder, {\em More or generators for the alternating and symmetric groups}, 
Quart. J. Math. Oxford Ser. (2), 32 (1981), 137--163.

\bibitem[FL1]{FL:symmetric}
M. Fernandes, D. Leemans, {\em Polytopes of high rank for the symmetric groups}, Adv. Math. 228 (2011), 3207--3222.

\bibitem[FL2]{FL2}
M. Fernandes, D. Leemans,  {\em String C-group representations of alternating groups}.
{\tt https://arxiv.org/abs/1810.12450}.

\bibitem[HL]{HL:Petrie}
M.I. Hartley, D. Leemans, {\em A new Petrie-like construction for abstract polytopes},
J. Combin. Theory, Ser. A 115 (2008), 997--1007.

\bibitem[LS]{LS:PSL2}
D. Leemans, E. Schulte, {\em Groups of type $L_2(q)$ acting on polytopes}, Adv. Geom. 7 (2007), no. 4, 529--539.

\bibitem[Ma]{Maz:sporadic}
V.D. Mazurov, {\em Generation of sporadic simple groups by three involutions, two of which commute},
Sibirsk. Mat. Zh. 44, no. 1 (2003), 193--198.

\bibitem[McS]{McS:ARP}
P. McMullen, E. Schulte, {\em Abstract Regular Polytopes}, Encyclopedia of Mathematics and its Applications 92, 
Cambridge University Press, 2002.

\bibitem[MS]{MS:reduction}
B. Monson, E. Schulte, {\em Reflection groups and polytopes over finite fields III}, Adv. in Appl. Math. 41, no. 1 (2008), 76--94.

\bibitem[Nu1]{Nu:even-rank3}
Y.N. Nuzhin, {\em Generating triples of involutions of Chevalley groups over a finite field of characteristic 2},
Algebra i Logika 29 (1990), 192--206.

\bibitem[Nu2]{Nu:alt-rank3}
Y.N. Nuzhin, {\em Generating triples of involutions of alternating groups},
Mat. Zametki 4 (1990), 91--95.

\bibitem[Nu3]{Nu:odd-rank3-1}
Y.N. Nuzhin, {\em Generating triples of involutions for Lie-type groups over a finite field of odd characteristic. I},
Algebra i Logika 36 (1997), 77--96.

\bibitem[Nu4]{Nu:odd-rank3-2}
Y.N. Nuzhin, {\em Generating triples of involutions for Lie-type groups over a finite field of odd characteristic. II},
Algebra i Logika 36 (1997), 422--440.

\bibitem[Tay]{Ta:geom_classical}
D.E. Taylor, {\em The geometry of classical groups}, Sigma series in pure mathematics 9, Heldermann Verlag, Berlin, 1992.

\end{thebibliography}

\end{document}